\documentclass{amsart}
\usepackage{amssymb}
\usepackage{amsmath}
\usepackage{hyperref}
\usepackage{color}
\newcommand{\bt}{\begin{theorem}}
\newcommand{\et}{\end{theorem}}
\newcommand{\bi}{\begin{itemize}}
\newcommand{\ei}{\end{itemize}}
\newcommand{\bea}{\begin{eqnarray}}
\newcommand{\ba}{\begin{array}}
\newcommand{\eea}{\end{eqnarray}}
\newcommand{\ea}{\end{array}}
\newcommand{\be}{\begin{equation}}
\newcommand{\ee}{\end{equation}}

\newcommand{\what}{\widehat}%
\newcommand{\wtilde}{\widetilde}%
\newcommand{\R}{\mathbb R}%
\newcommand{\C}{\mathbb C}%
\newcommand{\T}{\mathbb T}%
\newcommand{\N}{\mathbb N}%

\newtheorem{theorem}{Theorem}[section]

\newtheorem{lemma}[theorem]{Lemma}
\newtheorem{proposition}[theorem]{Proposition}

\theoremstyle{definition}

\theoremstyle{definition}
\newtheorem{remark}[theorem]{Remark}

\numberwithin{equation}{subsection}
\numberwithin{theorem}{subsection}

\begin{document}

\author[S. K. Ray]{Swagato K. Ray }
\address[S. K. Ray]{Department of Math. and Stat., Indian Institute of Technology, Kanpur 208016, India}
\email{skray@iitk.ac.in}

\author[R. P. Sarkar]{Rudra P. Sarkar}
\address[R. P. Sarkar]{Stat-Math Unit, Indian Statistical
Institute, 203 B. T. Rd., Calcutta 700108, India}
\email{rudra@isical.ac.in}

\subjclass[2000]{Primary 43A85; Secondary 22E30} \keywords{Wiener
class, groups of polynomial growth, $\mathrm{SL}(2, \R)$, symmetric
spaces.}

\title[Note on Kerman-Weit]{Note on a result of Kerman and Weit}
\maketitle

\begin{abstract} A result in \cite{Ker-Weit} states that a real
valued continuous function $f$ on the circle and its nonnegative integral powers
can generate a dense translation invariant subspace in the space of all continuous
functions on the circle if $f$ has a unique maximum or a unique minimum.
In this note we endeavour to show that this is quite a general phenomenon in harmonic analysis.
 \end{abstract}
\section{Introduction}
A recent paper of Kerman and Weit \cite{Ker-Weit} states that if a
continuous real valued function $f$ on the circle group $\T$ has
unique maximum or unique minimum, then the span of translations of
$f^n$ ($f^n(x)=(f(x))^n$), $n=0, 1, \dots$ is dense in the space of
all continuous functions  on $\T$. This note grew out as an effort
to understand the universality of this phenomenon. We bring into
fold a large class of groups and some  homogenous spaces to
establish that the argument in \cite{Ker-Weit} is robust enough to
be used {\em mutatis mutandis} in all these cases,  to yield similar
results involving $C_0$, $L^1$ and $L^p$ spaces. Precisely, we shall
consider the groups of Wiener class  (which includes all
locally compact abelian groups, compact topological groups, groups
of polynomial growth, $ax+b$ group etc.), as well as the group
$\mathrm{SL}(2, \R)$ which is not in the Wiener class and is of
exponential growth. We shall also take up  the Riemannian symmetric
spaces of compact and noncompact types.
\subsection{Notation} We shall use $\ell_x,
R_x$ to denote left and right translation by $x$ on a group $G$. For
a complex number $z$, its real and imaginary parts are denoted by
$\Re z$ and $\Im z$ respectively. For a set $A$ in a measure space
$|A|$ denotes its measure. The set of positive integers will be denoted by $\N$. For a
topological space $X$,  $C_c(X)$ (respectively $C_0(X)$) denotes the  space of continuous functions with compact support (respectively continuous functions vanishing at $\infty$ with supremum norm topology). For a group $G$, we denote its unitary dual by $\what{G}$.
\section{Analogues of the result of Kerman and Weit}
\subsection{The key lemma} To make it applicable to a wider context, in particular to accommodate the
noncompact spaces, we shall rewrite  the main argument of
\cite{Ker-Weit} with suitable modifications (see Lemma
\ref{meta-result} below). We notice that when the space $X$ is
noncompact, the function $f^0\equiv 1$ is not in $L^p(X)$ or in
$C_0(X)$. Therefore unlike \cite{Ker-Weit}, we shall consider only
the positive integral powers of $f$. Following \cite{Ker-Weit}, for
any $\lambda, \lambda_1, \lambda_2\in \R$, we define
$A_\lambda=\{x\mid f(x) \ge \lambda\}$ and $A_{\lambda_1,
\lambda_2}=\{x\mid \lambda_1< f(x)<\lambda_2\}$. We shall use the
following simple observation: If
 $T$ is a continuous function on an interval $[a, b]$ containing $0$ and
 $T(0)=0$, then there is a sequence of polynomials without constant terms $\{p_j\}$ such that
$\|p_j-T\|_{L^\infty[a, b]}\to 0, \text{  as } j\to \infty.$ Indeed
by Weierstrass approximation theorem we get a sequence $\{q_j\}$ of
polynomials such that $\|q_j-T\|_{L^\infty[a, b]}\to 0$ as $j\to
\infty$. Then in particular $c_j=q_j(0)\to T(0)=0$. Let
$p_j=q_j-c_j$. Then $p_j$ are without constant terms and
$\|p_j-T\|_{L^\infty[a, b]}=\|q_j-c_j-T\|_{L^\infty[a, b]}\le
\|q_j-T\|_{L^\infty[a, b]}+|c_j|\to 0$ as $j\to \infty$.
\begin{lemma} Let $X$ be a locally compact Hausdorff topological space equipped with a regular nonatomic Borel  measure $\mu$.
Let $f:X\to \R$ be a continuous function with compact support $S$,
such that  $f$ attains its maximum (or minimum) only at one point
say $x_0\in X$, which is not a isolated point of $X$.  If $\psi:
X\to \C$ is a continuous function and $\psi(x_0)\neq 0$, then there
exists a positive integer $n$ such that $\int_Xf^n(x)
g(x)d\mu(x)\neq 0$. \label{meta-result}
\end{lemma}
\begin{proof}
If $f$ has a unique minimum then $-f$ has a unique maximum. So it is
enough to consider the case that $f$ has a unique maximum. Since
maximum of $f$ is unique and $f$ is zero outside a compact set,
$\max f>0$. Normalizing we assume  that  $f(x_0)=1$. We fix a
$0<\eta<1$.  Then there is an open neighborhood $N_1$ of $x_0$ such
that  $f(x)>\eta$ for all $x\in N_1$.
\vspace{.1in}

\noindent{\em Step 1.}  We shall show in this step that there is a
$\lambda$ such that $\eta<\lambda<1$ and
$\int_{A_\lambda}\psi(x)dx\neq 0$. We start supposing that the
conclusion is false. Assume without loss of generality that $\Re
\psi(x_0)>0$. We take another open neighborhood $N_2$ of $x_0$ such
that for all $x\in N_2$, $\Re\, \psi(x)>0$. We define $N=N_1\cap
N_2$ which is an open neighbourhood of $x_0$.

Consider the compact set $S\setminus N=S\cap N^c$. Suppose $\max
f=\lambda_1$ in  the compact set $S\setminus N$. Then $\lambda_1<1$,
since maximum of $f$ is achieved only at $x_0$.

We choose $\varepsilon >0$ such that $\max\{\eta,
\lambda_1\}<1-\varepsilon$. Then $A_{1-\varepsilon, 1}=\{x\mid
1-\varepsilon<f(x)<1\}\subset N$. (Indeed $x\in A_{1-\varepsilon,
1}$ implies $f(x)>0$ and hence $x\in S$. If $x\not\in N$ then $x\in
S\setminus N$. Hence $f(x)\le \lambda_1<1-\varepsilon$ which
contradicts the fact that $x\in A_{1-\varepsilon, 1}$.) Since $x_0$
is not an isolated point $A_{1-\varepsilon, 1}\neq \emptyset$ and
hence $\mu(A_{1-\varepsilon, 1})> 0$. Since $\Re \psi(x)>0$ on $N$,
we have $\int_{A_{1-\varepsilon, 1}}\Re\, \psi(x)\,\, dx>0$. Now
\[\int_{A_{1-\varepsilon,1}}\psi(x)dx=
\int_{A_{1-\varepsilon}}\psi(x)dx-\int_{A_1}\psi(x)-\int_{f^{-1}(1-\varepsilon)}\psi(x)dx
=\int_{A_{1-\varepsilon}}\psi(x)dx
-\int_{f^{-1}(1-\varepsilon)}\psi(x)dx
\]
and \[\int_{f^{-1}(1-\varepsilon)}\psi(x)dx=\lim_{\mu\to
(1-\varepsilon)}\left(\int_{A_{1-\varepsilon}}\psi(x)dx-\int_{A_\mu}\psi(x)dx\right).\]
The assumption that for all $\lambda$ such that $\eta<\lambda<1$,
$\int_{A_\lambda}\psi(x)dx=0$ then shows that
$\int_{A_{1-\varepsilon,1}}\psi(x)dx=0$ which is a contradiction.
\vspace{.1in}

\noindent{\em Step 2.} By step 1  there is a $\lambda$ satisfying
$0<\eta<\lambda<1$ and $\int_{A_\lambda}\psi(x)dx\neq 0$. We
fix that $\lambda$. Let us assume that $\Re
\int_{A_\lambda}\psi(x)dx=q>0$. Let $b=\max_{x\in S} |\Re
\psi(x)|$. We choose a $0<\delta<1$ such that $\lambda-\delta>0$ and
$|A_{\lambda-\delta, \lambda}|<q/2b$.

Let $m=\min f$ (which can be negative). We consider a function $T\in
C[m,2]$ supported on $[\lambda-\delta, 1+\delta]$, $\max T=1$ and
$T(x)=1$ for $\lambda\le x\le 1$. We note that $T(0)=0$ as
$0<\lambda-\delta$. Then the function $T(f)(x)=T(f(x))$  is a
continuous function on $X$ and is supported on $S$.
 Now, \[\Re \int_X
T(f)(x)\psi(x)dx=\int_{A_{\lambda-\delta, \lambda}} T(f)(x)\Re
\psi(x)\, dx +\int_{A_\lambda} T(f)(x)\Re \psi(x)dx.\] Since
$|T(f)|\le 1$, absolute value of the first term in the right hand
side is $\le |A_{\lambda-\delta, \lambda}| b <q/2$. Since
$T(f)\equiv 1$ on $A_\lambda$,  the second term is equal to
$\int_{A_\lambda} \Re \psi(x)dx=q$. Hence $\Re \int_X
T(f)(x)\psi(x)dx>q/2$. \vspace{.1in}

\noindent{\em Step 3.} We shall show that there is $n\in \N$ such
that $\int_X f^n(x)\psi(x)dx\neq 0$. Suppose that $\int_X
f^n(x)\psi(x)dx=0$ for all $n\in \N$. Then for any polynomial $p$ in
one variable without constant term $p(f)$ is a continuous function
on $X$ supported on $S$ (as $p(0)=0$) and $\int_X p(f)(x)\psi(x)dx=0$.
By the observation above, there is a sequence $p_j$ of polynomials
without constant term such that $\|p_j-T\|_{L^\infty[m,2]}\to 0$ as
 $j\to \infty$ and hence as $j\to \infty$,
\[
|\int_X p_j(f)(x)\psi(x)dx-\int_X T(f)(x)\psi(x)dx| \le
C\|p_j(f)-T(f)\|_{L^\infty(S)} \le C\sup_{y\in [m, 2]}|p_j(y)-T(y)|
\to 0,
\]
where $C=\int_S|\psi(x)|dx$. This  implies $\int_X
T(f)(x)\psi(x)dx=0$.
\end{proof}

We shall use the lemma above to prove analogues of the result of
Kerman and Weit on various topological groups and homogenous spaces.
Groups where Wiener Tauberian theorem holds true seem to be the first natural target.
\subsection{Groups of Wiener class}
A locally compact group $G$ equipped with a left Haar measure  is in
Wiener class ([W]-class for short) if every proper two-sided ideal
in the Banach$^\ast$-algebra
 $L^1(G)$ is annihilated by a nondegenerate continuous $^\ast$-representation
 $\pi$ of $L^1(G)$ in a Hilbert space (see \cite{Lep}). Thus if $G$ is of [W]-class then a family of functions $\{f_\alpha\}$
 in $L^1(G)$ can generate a dense two-sided ideal in $L^1(G)$ if and only if for any $\pi\in \what{G}$,
 there exists $f_\alpha$ in that collection such that $\what{f_\alpha}(\pi)=\int_G f_\alpha(x)\pi(x)dx\neq 0$.
We recall  that the closure in $L^1(G)$ of the two-sided ideal generated by  $\{f_\alpha\}$ is same as the closure of the span of left and right translates of $\{f_\alpha\}$.

\begin{theorem}
Let $G$ be a group of class {\em [W]}. Let $f\in C_c(G)$ be a real valued
function which has a unique maximum or a unique minimum at a point $x_0\in
G$. Then the both-sided ideal generated by $\{f^n\mid n\in \N\}$ is
dense in $L^1(G)$.
\end{theorem}
\begin{proof} We first observe that if $G$ is discrete then Lemma
\ref{meta-result} is not applicable. However it does not take much
effort to prove an analogue independently for discrete groups.

Since $G$ is of class [W], it suffices to show that for any $\pi\in
\what{G}$, there exists $n\in \N$ and $u, v\in H_\pi$ such that
$\what{f^n}(\pi)_{u, v}=\langle \what{f^n}(\pi)u,
v\rangle=\int_Gf^n(x)\langle \pi(x)u, v\rangle d\mu(x)\neq 0.$
Since the operator $\pi(x_0)$ is unitary, we can choose $u, v\in
H_\pi$ such that $\langle \pi(x_0)u, v\rangle\neq 0$.
We define $\psi(x)=\langle \pi(x)u, v\rangle$  and  use Lemma
\ref{meta-result}.
\end{proof}

\begin{remark}
If  $G$ is a locally compact  abelian group   then in the theorem
above, $L^1(G)$ can be replaced by $L^p(G), 1<p<\infty$ or by
$C_0(G)$. We sketch the argument below for $L^p(G)$. We fix a $p\in
[1, \infty)$. Let $p'=p/(p-1)$. If a function $h\in L^{p'}(G)$
satisfies $\int_{G}\ell_x f^n(y)h(y)dy=0$ for all $n\in \N$ and all
$x\in G$ then $f^n\ast h\equiv 0$ for all $n\in \N$. Let
$\{\phi_m\}$ be an approximate identity in $L^1(G)\cap L^\infty(G)$.
Then $h\ast \phi_m\in L^{p'}(G)\cap L^\infty(G)$ and $f^n\ast(h\ast
\phi_m)\equiv 0$ for all $n\in \N$. Since by the theorem above
$\mathrm{span }\{\ell_x f^n\mid n\in \N, x\in G\}$ is dense in
$L^1(G)$, we have $h\ast \phi_m=0$ for all $m$ and hence $h=0$.
\end{remark}

\subsection{A group of exponential growth} For this subsection $G=\mathrm{SL}(2, \R)$. Our aim here is the following.
\begin{theorem} Let $f\in C_c(G)$ be a real valued
function which has a unique maximum or a unique minimum at a point
$x_0\in G$. Then  $\mathrm{span }\{\ell_xR_yf^n\mid x,y\in G, n\in
\N\}$ is dense in $L^p(G), 1\le p<2$. \label{sl2}
\end{theorem}
We recall  that $G$  is a prototype of the noncompact semisimple Lie
group.  Since the seminal work of Ehrenpreis and Mautner \cite{EM},
it is known  that this group is not in the [W]-class. However a
modified version of the Wiener's Tauberian theorem for $L^p(G), 1\le
p<2$  is available in the literature (see \cite{EM, BW, sarkar} and
the references therein). The main difference is that for
$\mathrm{SL}(2, \R)$ (or generally for noncompact semisimple Lie
groups) one cannot rely only on the unitary dual of the group. We
state here a version of the theorem suitable for our present need.
This is indeed a corollary of \cite[Theorem 1.2]{sarkar} which   can
be verified easily.
 \begin{proposition} Fix $1\le p<2$ and choose a $q<p$.
 Let $\mathcal L\subset C_c(G)$ be a collection of functions such that
 their Fourier transforms with respect to the irreducible $L^q$-tempered
 representations
 do not have a common zero. Then $\mathrm{span }\{\ell_xR_yf\mid x,y\in G, f\in \mathcal L\}$ is
dense in $L^p(G)$.
 \end{proposition}
 We recall that for $0<p\le 2$, a representation is  $L^p$-tempered  when   the Fourier
 transform  of a function in the  $L^p$-Schwartz space  on $G$ exists
 with respect to this representation. $0<p<2$, depending on $p$,
 the set of $L^p$-tempered representations are subsets of the union of the following two sets of representations:
 (1) unitary principal series, discrete and mock discrete series
 representations and the trivial representation, (2) certain nonunitary
 principal series representations. For a detailed account we refer to \cite{Ba, sarkar}. Presently we only need
 the fact that all these representations
 (except the trivial one) are infinite dimensional and realized on separable  Hilbert spaces.
In view of  proposition above it suffices to show that  the operator
valued Fourier transforms of the functions $f^n, n\in \N$ with
respect to these representations do not have a common zero.  For any
such representation $(\pi, H_\pi)$, we look for a matrix coefficient
$\psi(x)=\langle\pi(x)v, w\rangle$ such that $\psi(x_0)\neq 0$. An
applciation of Lemma \ref{meta-result} with this $\psi$ proves
Theorem \ref{sl2}.

\subsection{Riemannian Symmetric spaces} In this subsection we shall
deal with the Riemannian Symmetric spaces of compact and noncompact type in
that order.

Let $X=G/K$ be a Riemannian symmetric space of compact type where
$G$ is a compact connected semisimple Lie group and $K$ be the fixed
point group of an involutive automorphism on $G$. Then functions on
$X$ can be realized as right $K$-invariant functions on $G$ and
vice-versa. For $(\delta, V_\delta)\in \what{G}$, let
$V_\delta^K=\{v\in V_\delta\mid \delta(k)v=v\,\,\, \forall k\in
K\}$. We recall that $\dim V_\delta^K\le 1$ (see
\cite[p.412]{Helga-GGA}). Let $\what{G}_K=\{\delta\in \what{G}\mid
\dim V_\delta^K=1\}$. It is easy to verify  that if $f$ is a
continuous complex valued function on $X$ (equivalently if $f$ is a
right $K$-invariant function on $G$), then $\what{f}(\delta)=0$ if
$\delta\not\in \what{G}_K$. Given a $\delta\in\what{G}_K$, let
$e_0\in V_\delta^K$ and $\{e_0, e_1, \dots, e_{d(\delta)}\}$ be an
orthonormal basis of $V_\delta$, where $d(\delta)=\dim V_\delta$.
Using Weyl's orthoganility relation it can be shown that $\langle
\what{f}(\delta)e_i, e_j\rangle=\int_G f(x)\langle\delta(x)e_i,
e_j\rangle dx =0$ unless $i=0$. Similarly, if $f$ is a left
$K$-invariant function on $G$ then $\langle \what{f}(\delta)e_i,
e_j\rangle =0$ unless $j=0$. It is easy to see that the above
conclusions also apply to complex measures on $G$.
\begin{theorem}
Let $X=G/K$ be a  Riemannian symmetric space of compact type and $f$
be a a real valued continuous function which has a unique maximum or
a unique minimum at $x_0\in X$. Then $\mathrm{span }\{\ell_x f^n\mid
n\in \N, x\in G\}$ is dense in $C(X)$.
\end{theorem}
\begin{proof}
Let $\mu$ be a nonzero complex right $K$-invariant measure on $G$,
such that $\mu(\ell_xf^n)=0$ for all $x\in G$ and for all $n\in \N$.
Then $f^n\ast \wtilde{\mu}\equiv 0$ where $\wtilde{\mu}(g)=\int
g(x^{-1})d\mu(x)$ for all $g\in C(G/K)$. It is clear that
$\wtilde{\mu}$ is left $K$-invariant. Consequently, from the
discussion above it follows that for all $n\in \N$,  $e_i, e_j\in
V_\delta$ and $\delta\in \what{G}_K$,
\[\langle(f^n\ast\wtilde{\mu})\what{}(\delta)e_i, e_j\rangle=\langle\what{\wtilde{\mu}}
(\delta) e_i, e_0\rangle \langle \what{(f^n)^\ast}(\delta)e_j,
e_0\rangle=0.\] Since there exists $\delta\in \what{G}_K$ and
$e_i\in V_\delta$ such that $\langle\what{\wtilde{\mu}} (\delta)
e_i, e_0\rangle\neq 0$, $\langle \what{(f^n)}(\delta)e_0,
e_j\rangle=0$ for all $e_j\in V_\delta$ and $n\in \N$.

Thus it is enough to show that for some $n\in \N$ and $e_j\in
V_\delta$,  $\langle \what{(f^n)}(\delta)e_0, e_j\rangle\neq 0$. We
note that for some $e_j\in V_\delta$, $\langle \delta(x) e_0,
e_j\rangle\neq 0$. We fix that $e_j$ and  define $\psi(x)=\langle
\delta(x) e_0, e_j\rangle$ on $G$ and apply Lemma \ref{meta-result}
with this $\psi$ to complete the proof.
\end{proof}

Next we shall take up the case of Riemannian symmetric spaces of
noncompact type. The notation we shall use here is standard and can
be found in \cite{Helga-GGA}. We shall mention here only those which
are required to describe the result. Let $G$ be connected noncompact
semisimple Lie group with finite centre and $K$ be a maximal compact
subgroup of $G$. Then the homogenous space $X=G/K$ is a Riemannian
symmetric space of noncompact type. Three other important subgroups
of $G$ are $A$ which is isomorphic with $\R^n$, $N$ which is a
nilpotent group and $M$ which is the centralizer of $A$ in $K$.
Every element $g\in G$ can be uniquely written as $g=k(g)H(g)n(g)$
where $k(g)\in K$, $H(g)\in \mathfrak a$, the Lie algebra of $A$ and
$n(g)\in N$. In this note we shall deal only with rank one symmetric
spaces of noncompact type, i.e. $\dim \mathfrak a=1$. Indeed we
shall identify $\mathfrak a$ with $\R$. The group $G$ acts on $X$
naturally by left translation. The functions on $X$ are identified
with right $K$-invariant functions on $G$ and vice-versa. There
exists a $G$-invariant measure on $X$ with respect to which we
consider the $L^p$-spaces on $X$. Let $\rho$ be the half-sum of the
positive roots  which will be considered as a scaler as $\dim
\mathfrak a=1$. For a function $f\in C_c(X)$  we define its Helgason
Fourier transform $\wtilde{f}(\lambda, k)$ for $k\in K$ and
$\lambda\in \C$ by (see \cite{Helga-GGA}),
\[\wtilde{f}(\lambda, k)=\int_G f(x) e^{-(i\lambda+\rho)H(x^{-1}k)}dx.\]

\begin{theorem}\label{noncom-sym}
Let $X=G/K$ be a rank one Riemannian symmetric space of noncompact
type and $f\in C_c(X)$ be a real valued function which has a unique
maximum or a unique minimum at $x_0\in X$. Then $\mathrm{span
}\{\ell_x f^n\mid n\in \N, x\in G\}$ is dense in $L^1(X)$.
\end{theorem}
\begin{proof} Since $f\in C_c(X)$ it is clear that
$\wtilde{f}(\lambda, k)$ is a continuous function in  $(\lambda,
k)\in\C\times K/M$ and for any fixed $k\in K/M$, $\lambda\mapsto
\wtilde{f}(\lambda, k)$ extends holomorphically to $\C$ and in
particular to $S_1^\varepsilon=\{z \in \C\mid |\Im
z|<\rho+\varepsilon\}$ for any $\varepsilon>0$. We fix a
$\varepsilon>0$. In view of \cite[Theorem 5.5]{MRSS} we need to
verify that
\begin{enumerate}
\item[(i)] for any $\lambda\in S_1^\varepsilon$, there exists $k\in K$ such that $\wtilde{f^n}(\lambda, k)\neq 0$ for some $n\in
\N$ and
\item[(ii)] $f$ is not real analytic.
\end{enumerate}
Condition (ii) is obvious as $f$ is compactly supported.

Since  $e^{(i\lambda+\rho)H(x_0^{-1}k)}\neq 0$ for any $k\in K/M$,
we fix a $k\in K/M$  and consider this as
$\psi(x)=e^{(i\lambda+\rho)H(x_0^{-1}k)}$ in the hypothesis of Lemma
\ref{meta-result}. This takes care of (i).
\end{proof}
An analogous result can be proved for $L^p(X), 1<p<2$ using
\cite[Theorem 5.7]{MRSS}, which we omit for brevity.
\section{Examples and Comments}
We shall consider three examples in three different set up where the
functions attain their maximum at more than one points.

(1) We take a function  $f_1\in C_c(\R^n)$. Let $f_2$ be a
translation of $f_1$ by $a\in \R^n$ so that support of $f_1$ and
$f_2$ are disjoint. Let $f=f_1+f_2$. Then $f\in C_c(\R^n)$ has
non-unique maximum and minimum. Note that for all $n\in \N$, $f^n=
f_1^n+f_2^n$, $f_2^n$ is the translation of $f_1^n$ by $a$ and
$\what{f^n}(\xi)=\what{f_1^n}(\xi) + e^{i\xi\cdot
a}\what{f_1^n}(\xi)$. We take  a $\xi$ so that $e^{i\xi\cdot a}=-1$.
Then $\what{f^n}(\xi)=0$ for all $n\in \N$. Hence $\mathrm{span
}\{\ell_x f^n\mid n\in \N, x\in \R^n\}$ is not dense in $L^1(\R^n)$.

(2) Let $X=G/K$ be a Riemannian symmetric space of noncompact type.
We take a function $f\in C_c(G/K)$ which has unique maximum at the
point $x_0$. By Theorem \ref{noncom-sym}, $\{f^n\mid n\in \N\}$
generates  $L^1(G/K)$. But $x_0$ considered as a point on $G/K$ is
actually a coset $x_0=gK$ for some $g\in G$. Thus considered as a
function of $G$, $f$  has maximum at all  points of the orbit
$\{gk\mid k\in K\}$. If for a $\pi\in \what{G}$, $\pi|_K$ does not
contain the trivial representation of $K$, then $\what{f^n}(\pi)=0$
for all $n\in \N$. Therefore the closure of $\mathrm{span }\{\ell_x
R_y f^n\mid n\in\N, x, y\in G\}$ contains only the functions whose
Fourier transforms vanish at $\pi$.

(3) We take $G=\mathrm{SL}(2, \R)$. Let $I$ be the  $2\times 2$
identity matrix. Suppose that $f\in C_c(G)$ satisfies $f(x)=f(-x)$
for all $x\in G$ where  $-x=-I x$ and attains its maximum at $\pm
x_0$. Then for any $n\in \N$, $f^n$ can be considered as a function
of $\mathrm{SL}(2, \R)/\{\pm I\}$ and the closure in $L^1(G)$ of the
span of translates of $f^n$ contains only  those  functions which
take same value at $\pm x$.  Thus $\mathrm{span }\{\ell_x R_y
f^n\mid n\in \N, x, y\in G\}$ cannot be dense in  $L^1(G)$. On the
other hand if  $f$ is considered as a function of $\mathrm{SL}(2,
\R)/\{\pm I\}$, then it has a unique maximum and  argument similar
to that of subsection 2.3 shows that $\{f^n\mid n\in \N\}$ generates
$L^1(\mathrm{SL}(2, \R)/\{\pm I\})$.


\begin{thebibliography}{99}
\bibitem{Ba}Barker, W. H. {\em $L^p$ harmonic analysis on $\mathrm{SL}(2,\R)$}. Mem. Amer. Math. Soc. 76 (1988), no.
393.
\bibitem{BW} Benyamini, Y.; Weit, Y. {\em Harmonic analysis of spherical functions on ${\mathrm SU}(1,1)$}. Ann. Inst. Fourier (Grenoble) 42 (1992), no. 3, 671–-694.
\bibitem{EM} Ehrenpreis, L.; Mautner, F. I. {\em Some properties of the Fourier transform on semi-simple Lie groups. I}. Ann. of Math. (2) 61, (1955). 406–-439.
\bibitem{Helga-GGA} Helgason, S. {\em Groups and geometric analysis. Integral geometry, invariant differential operators, and spherical functions}.
 Math. Surveys and Monographs, 83. AMS, Providence, RI, 2000.
\bibitem{HJLP} Hulanicki, A.; Jenkins, J. W.; Leptin, H.; Pytlik, T. {\em Remarks on Wiener's Tauberian theorems for groups with polynomial growth}. Colloq. Math. 35 (1976), no. 2, 293–-304.
\bibitem{Ker-Weit} Kerman, ~R. ~A.; Weit, ~Y. {\em On the translates of powers of a continuous periodic function}. J. Fourier Anal. Appl. 16 (2010), no. 5, 786-–790.
\bibitem{Lep} Leptin, H. {\em Ideal theory in group algebras of locally compact groups}. Invent. Math. 31 (1975/76), no. 3, 259–-278.
\bibitem{MRSS}  Mohanty, ~P.; Ray, ~S. ~K.; Sarkar, ~R. ~P.; Sitaram, ~A.
{\em The Helgason-Fourier transform for symmetric spaces. II}. J.
Lie Theory {\bf 14} (2004), no. 1, 227--242.
\bibitem{sarkar} Sarkar, R. P. {\em Wiener Tauberian theorems for $\mathrm{SL}(2, \R)$}. Pacific J. Math. 177 (1997), no. 2, 291-–304.
\end{thebibliography}
\end{document}